\documentclass[11pt]{article}
\usepackage{booktabs} 
\usepackage{fullpage, times}
\usepackage{times}

\usepackage[utf8]{inputenc} 
\usepackage[T1]{fontenc}    
\usepackage{hyperref}       
\usepackage{url}            
\usepackage{booktabs}       
\usepackage{amsfonts}       
\usepackage{nicefrac}       
\usepackage{microtype}      

\usepackage{enumitem}

\usepackage{times,url,mathtools}

\usepackage{amsthm,amsfonts,amsmath,amssymb,epsfig,color,float,graphicx,verbatim}
\usepackage{algorithm,algorithmic}

\newtheorem{theorem}{Theorem}

\newtheorem{lemma}{Lemma}

\newcommand{\reals}{\mathbb{R}}
\newcommand{\E}{\mathbb{E}}

\newcommand{\be}{\mathbf{e}}

\newcommand{\bw}{\mathbf{w}}

\newcommand{\bb}{\mathbf{b}}

\newcommand{\by}{\mathbf{y}}

\newcommand{\bxi}{\boldsymbol{\xi}}

\newcommand{\bepsilon}{\boldsymbol{\epsilon}}

\newcommand{\Ocal}{\mathcal{O}}

\newcommand{\norm}[1]{\|#1\|}
\newcommand{\inner}[1]{\langle#1\rangle}

\newcommand{\secref}[1]{Sec.~\ref{#1}}

\renewcommand{\eqref}[1]{Eq.~(\ref{#1})}
\newcommand{\lemref}[1]{Lemma~\ref{#1}}

\newcommand{\thmref}[1]{Thm.~\ref{#1}}

\title{A Tight Convergence Analysis for Stochastic Gradient 
	Descent with Delayed Updates}

\author{
  Yossi Arjevani \hspace{2.5cm} Ohad Shamir\\
	Weizmann Institute of Science\\
	Rehovot 7610001, Israel \\
  \texttt{\{yossi.arjevani,ohad.shamir\}@weizmann.ac.il} \\
   \and
   \and
   Nathan Srebro \\
   TTI Chicago \\
   Chicago, IL 60637\\
   \texttt{nati@ttic.edu} \\
}

\date{}
\begin{document}

\maketitle

\begin{abstract}
	We provide tight finite-time convergence bounds for gradient descent and 
	stochastic gradient descent on quadratic functions, when the gradients are 
	delayed and reflect iterates from $\tau$ rounds ago. First, we show that 
	without stochastic noise, delays strongly affect the attainable 
	optimization error: In fact, the error can be as bad as non-delayed 
	gradient descent ran on only $1/\tau$ of the gradients. In sharp contrast, 
	we quantify how stochastic noise makes the effect of delays negligible, 
	improving on previous work which only showed this phenomenon asymptotically 
	or for much smaller delays. Also, in the context 
	of distributed optimization, the results indicate that the performance of 
	gradient descent with delays is competitive with synchronous approaches 
	such as mini-batching. Our results are based on a novel technique for 
	analyzing convergence of optimization algorithms using generating functions.
\end{abstract}

\section{Introduction}

Gradient-based optimization methods are widely used in machine learning and 
other large-scale applications, due to their simplicity and scalability. 
However, in their standard formulation, they are also strongly synchronous and 
iterative in nature: In each iteration, the update step is based on the 
gradient at the current iterate, and we need to wait for this 
computation to finish before moving to the next iterate. For example, to 
minimize some function $F$, plain stochastic gradient descent initializes at 
some point $\bw_0$, and computes iterates of the form
\begin{equation}\label{eq:sgd}
\bw_{k+1} = \bw_{k}-\eta (\nabla F(\bw_k)+\bxi_k)~,
\end{equation}
where $\nabla F(\bw_k)$ is the gradient of $F$ at $\bw_k$, $\eta$ is the step 
size and $\bxi_1,\bxi_2,\ldots$ are independent zero-mean noise terms. 
Unfortunately, in several important applications, a direct implementation of 
this is too costly. For example, consider a setting where we wish to optimize a 
function $F$ using a distributed platform, consisting of several machines with 
shared memory. We can certainly implement gradient descent, by letting one of 
the machines compute the gradient at each iteration, but this is clearly 
wasteful, since just one machine is non-idle at any given time. Thus, it is 
highly desirable to use methods which parallelize the computation. One approach 
is to employ \emph{mini-batch gradient} methods, which parallelize the 
computation of the stochastic gradient, and their analysis is relatively well 
understood (e.g. 
\cite{dekel2012optimal,cotter2011better,shamir2014distributed,takac2013mini}). 
However, these methods are still generally iterative and synchronous in 
nature, and hence can suffer from problems such as having to wait for the 
slowest machine at each iteration. 

A second and popular approach is to utilize  
\emph{asynchronous} gradient methods. With these methods, each update step is 
not necessarily based just on the 
gradient of the current iterate, but possibly on the gradients of earlier 
iterates (often called \emph{stale updates}). For example, when optimizing a 
function using several machines, each machine might read the current iterate 
from a shared parameter server, compute the gradient at that iterate, and then 
update the parameters, even though other machines might have performed other 
updates to the parameters in the meantime. Although such asynchronous methods 
often work well in practice, analyzing them is much trickier than synchronous 
methods. 

In our work, we focus on arguably the simplest possible variant of these 
methods, where we perform plain stochastic gradient descent on a convex 
function $F$ on $\reals^d$, with a fixed delay of $\tau>0$ in the gradient 
computation:
\begin{equation}\label{eq:dsgd}
\bw_{k+1} = \bw_{k}-\eta (\nabla F(\bw_{k-\tau})+\bxi_k)~,
\end{equation}
where we assume that $\bw_0=\bw_1=\ldots=\bw_{\tau}$. 
Compared to \eqref{eq:sgd}, we see that the gradient is computed with respect 
to $\bw_{k-\tau}$ rather than $\bw_k$. Already in this simple formulation, 
the precise effect of the delay on the convergence rate is not completely 
clear. For example, for a given number of iterations $k$, how large can $\tau$ 
be before we might expect a significant deterioration in the accuracy? And 
under what conditions? Although there exist some prior results in this 
direction (which we survey in the related work section below), these questions 
have remained largely open. 

In this paper, we aim at providing a tight, finite-time convergence analysis 
for stochastic gradient descent with delays, focusing on 
the simple case where $F$ is a convex quadratic function. Although a quadratic 
assumption is non-trivial, it arises naturally in problems such as least 
squares, and is an important case study since all smooth and convex function 
are locally quadratic close to their minimum (hence, our results should still 
hold in a local sense). In future work, we hope to show that our results are 
also applicable more generally. 

First, we consider the case of \emph{deterministic} delayed gradient descent 
(DGD, defined in \eqref{eq:dsgd} with  $\bxi_k=\mathbf{0}$). Assuming the step 
size $\eta$ is chosen appropriately, we prove that
\begin{align*}
F(\bw_k)-F(\bw^*)~\leq~ 5\mu\|\bw_0 - \bw^*\|^2\exp\left(-\frac{\lambda 
(k+1)}{10\mu(\tau+1)}\right)
\end{align*}
after $k$ iterations, over the class of $\lambda$-strongly convex 
$\mu$-smooth quadratic functions with a minimum at $\bw^*$, and
\begin{align*}
F(\bw_k)-F(\bw^*)~\leq~\frac{17\mu\|\bw_0 - \bw^*\|^2(\tau+1)}{k+1}
\end{align*}
over the class of $\mu$-smooth  convex quadratic functions with minimum at 
$\bw^*$. In terms of 
iteration complexity, the number of iterations $k$ required to achieve a fixed 
optimization error of at most $\epsilon$ in the strongly convex and the 
convex cases is therefore
\begin{equation}
\Ocal\left(\tau\cdot\kappa\ln\left(\frac{\mu\norm{\bw_0-\bw^*}^2}{\epsilon}\right)\right)~~~~
\text{and}~~~~
\Ocal\left(\tau\cdot \frac{\mu\norm{\bw_0-\bw^*}^2}{\epsilon}\right)
\label{eq:itercomp}
\end{equation}
respectively, where $\kappa\coloneqq \mu/\lambda$ is the so-called 
condition number\footnote{Following standard 
convention, we use here the $\Ocal$-notation to hide 
constants, and tilde $\tilde{\Ocal}$-notation to hide constants and factors 
polylogarithmic in the problem parameters.}. When $\tau$ is a bounded constant, 
these bounds match the 
known iteration complexity of standard gradient descent without delays 
\cite{nesterov2004introductory}. 
However, as $\tau$ increases, both bounds deteriorate linearly with 
$\tau$. Notably, in our setting of delayed gradients, this implies that DGD 
is no better than a trivial algorithm, which performs a single gradient 
step, and then waits for $\tau$ rounds till the delayed gradient is 
received, before performing the next step (thus, the algorithm is 
equivalent to non-delayed gradient descent with $k/\tau$ gradient steps, 
resulting in the same linear deterioration of the iteration complexity with
$\tau$). 

Despite these seemingly weak guarantees, we show that they are in fact 
tight in terms of $\tau$, by proving that this linear dependence on $\tau$ is 
unavoidable with standard gradient-based methods (including gradient 
descent). The dependence on the other problem parameters in our lower bounds is 
a bit weaker than our upper bounds, but can be matched by an \emph{accelerated} 
gradient descent procedure (see \secref{section:deterministic_delayed}  
for more details). 

In the second part of our paper, we consider the case of \emph{stochastic} 
delayed gradient descent (SDGD, defined in (\ref{eq:dsgd})). Assuming $\bxi_k$ 
satisfies $\E[\norm{\bxi_k}^2]\leq \sigma^2$ and that the step size $\eta$ is 
appropriately tuned, we prove that
\begin{equation}
\E\left[F(\bw_k)-F(\bw^*)\right]~\leq~\tilde{\Ocal}\left(\frac{\sigma^2}{\lambda
k}+\mu \norm{\bw_0-\bw^*}^2\exp\left(-\frac{\lambda k
}{10\mu\tau}\right)\right)~.
\label{eq:sdgdstrong}
\end{equation}
for $\lambda$-strongly convex, $\mu$-smooth quadratic functions with minimum at 
$\bw^*$, and
\begin{equation}
\E\left[F(\bw_k)-F(\bw^*)\right]~\leq~ \tilde{\Ocal}\left(\frac{ 
\norm{\bw_0-\bw^*}\sigma}{\sqrt{k}}
+  \frac{\norm{\bw_0-\bw^*}^2 \mu\tau }{k}\right).
\label{eq:sdgdconvex}
\end{equation}
for $\mu$-smooth convex quadratic functions. In terms of iteration complexity, 
these correspond to
\begin{equation}
\tilde{\Ocal}\left(\frac{\sigma^2}{\lambda 
\epsilon}+\tau\cdot 
\kappa\ln\left(\frac{\mu\norm{\bw_0-\bw^*}^2}{\epsilon}\right)\right)
~~~\text{and}~~~
\tilde{\Ocal}\left(\frac{\norm{\bw_0-\bw^*}^2\sigma^2}{\epsilon^2}+\tau\cdot
\frac{\norm{\bw_0-\bw^*}^2\mu}{\epsilon}
\right)~,
\label{eq:itercompstoch}
\end{equation}
in the strongly convex and convex cases respectively, where again 
$\kappa:=\mu/\lambda$. As in the deterministic case, when $\tau$ is a 
bounded constant, these bounds match the known iteration complexity bounds for 
standard gradient descent without delays 
\cite{bubeck2015convex,shamir2013stochastic}. Moreover, these bounds match 
the bounds for the deterministic case in \eqref{eq:itercomp} when $\sigma^2=0$ 
(i.e. zero noise), as they should. However, in sharp contrast to the 
deterministic case, the dependence on $\tau$ 
in \eqref{eq:itercompstoch} is quite different: The delay $\tau$ only appears 
in second-order terms (as $\epsilon\rightarrow 0$), and its influence becomes 
negligible when $\epsilon$ is small enough. The same effect can be seen in 
\eqref{eq:sdgdstrong} and \eqref{eq:sdgdconvex}: Once the number of iterations 
$k$ is large enough, the first term in both bounds dominates, and $\tau$ no 
longer plays a role. More specifically:
\begin{itemize}[leftmargin=*]
	\item In the strongly convex case, the effect of the delay becomes 
	negligible once the target accuracy $\epsilon$ is sufficiently smaller than 
	$\tilde{\Ocal}(\sigma^2/(\mu\tau))$, or when the number of iterations $k$ 
	is sufficiently larger 
	than $\tilde{\Omega}(\tau\mu/\lambda)$. In other words, assuming the 
	condition number $\mu/\lambda$ is bounded, we can have the delay $\tau$ 
	nearly as large as the total number of iterations $k$ (up to 
	log-factors), without significant deterioration in the convergence 
	rate. Note that this is a mild requirement, since if $\tau\geq k$, the 
	algorithm receives no gradients and makes no updates.
	\item In the convex case, the effect of the delay becomes negligible once 
	the target accuracy $\epsilon$ is sufficiently smaller than
	$\tilde{\Ocal}(\sigma^2/(\mu\tau))$, or 
	when the number of iterations $k$ is sufficiently larger than $\tilde{\Omega}((\norm{\bw_0-\bw^*}\mu\tau/\sigma)^2)$. Compared to the 
	strongly convex case, here the regime is the same in terms of $\epsilon$, 
	but the regime in terms of $k$ is more restrictive: We need $k$ to scale 
	quadratically (rather than linearly) with $\tau$. Thus, the maximal delay 
	$\tau$ with no performance deterioration is order of $\sqrt{k}$.
\end{itemize}

Finally, it is interesting to compare our bounds to those of \emph{mini-batch} 
stochastic gradient descent (SGD), which can be seen as a synchronous 
gradient-based method to cope with delays, especially in distributed 
optimization and learning problems 
\cite{dekel2012optimal,cotter2011better,agarwal2011distributed}. In 
mini-batch SGD, each update step is performed only 
after accumulating and averaging a mini-batch of $b$ stochastic gradients, all 
with respect to the same point:
\[
\forall k\in \{0,b,2b,\ldots\},~~\bw_{k+b} = \bw_k-\eta\cdot 
\frac{1}{b}\sum_{i=0}^{b-1}\left(\nabla 
F(\bw_k)+\xi_{k+i}\right)~,
\]
Although the algorithm makes an update only every $b$ stochastic gradient 
computations, the averaging reduces the stochastic noise, and helps speed up 
convergence. Moreover, this can be seen as a particular type of algorithm with 
delayed updates (with the delay correspond to $b$), as we use $\nabla F(\bw_k)$ 
to compute iterate $\bw_{k+b}$. The important difference is that it is an 
inherently synchronous method, that waits for all $b$ stochastic gradients to 
be computed before performing an update step. Remarkably, the bounds we proved 
above for delayed SGD are essentially identical to those known for mini-batch 
SGD, with the delay $\tau$ replaced by the mini-batch size $b$ (at least in the 
convex case where mini-batch SGD has been more thoroughly analyzed). This 
indicates 
that an asynchronous method like delayed SGD can potentially match the 
performance of synchronous methods like mini-batch SGD, even without requiring 
synchronization -- an important practical advantage.

Analyzing gradient descent with delays is notoriously tricky, due to the 
dependence of the updates on iterates produced many iterations ago. The 
technique we introduce for deriving our upper bounds is primarily based on 
\emph{generating functions}, and might be useful for studying other 
optimization algorithms. 
We discuss this approach 
more thoroughly in Section \ref{section:generating_functions_and_framework}. 
The rest of the paper is devoted mostly to presenting the 
formal theorems and an explanation of how they are derived (with technical 
details relegated to the supplementary material). 


\subsection*{Related Work}

There is a huge literature on asynchronous versions of gradient-based methods 
(see for example the seminal book \cite{bertsekas1989parallel}), 
including treating the 
effect of delay. However, most of these do not consider the setting we study 
here. For example, there has been much recent interest in asynchronous 
algorithms, in a model where there is a delay in updating individual 
\emph{coordinates} in a shared parameter vector (e.g., the Hogwild! algorithm 
of \cite{recht2011hogwild}, or more recently 
\cite{mania2015perturbed,leblond18improved}). Of course, this is a different 
model than ours, where the updates use a full gradient 
vector. Other works (such as 
\cite{sirb2016decentralized}) focus on a setting where different agents in a 
network can perform local communication, which is again a different model than 
ours. Yet other works focus on sharp but asymptotic results, and do not provide 
guarantees after a fixed number $k$ of iterations
(e.g., \cite{chaturapruek2015asynchronous}). 

Moving closer to our setting, 
\cite{nedic2001distributed} showed convergence for delayed 
gradient descent, with the result implying an 
$\sqrt{\tau/k}$ convergence rate for convex functions. A similar bound on 
average regret has been shown in an adversarial online learning setting, for 
general convex functions, and this bound is known to be optimal 
\cite{joulani2013online}. These results differ from our setting, in that they 
consider possibly non-smooth functions, in which the dependence on 
$k$ is no better than $1/\sqrt{k}$ even without delays and no noise, and where 
the delay $\tau$ always plays a significant role. In contrast, we focus 
here on smooth functions, where rates better than $1/\sqrt{k}$ are possible, 
and where the effect of $\tau$ is more subtle. In 
\cite{feyzmahdavian2014delayed}, the authors study a setting very 
similar to ours in the deterministic case, and manage to prove a linear 
convergence rate, but for a less standard algorithm, different than the one we 
study here (with iterates of the 
form $\bw_{t+1} = \bw_{t-\tau}-\nabla 
F(\bw_{t-\tau})$). 

Perhaps the works closest to ours are 
\cite{agarwal2011distributed,feyzmahdavian2016asynchronous}, which 
study stochastic gradient descent with delayed gradients. Moreover, they 
consider a setting more general  than ours, where the delay at each 
iteration is any integer up to $\tau$ (rather than fixed $\tau$), and the 
functions are not necessarily quadratic. On the flip side, their bounds are 
significantly weaker. For example, for 
smooth convex functions and an appropriate step size, 
\cite[Corollary 1]{agarwal2011distributed} show a bound of
\[
\Ocal\left(\frac{\sigma}{\sqrt{k}}+\frac{\tau^2+1}{\sigma^2 
	k}\right).
\]
in terms of $k,\tau,\sigma$. Note that this bound is vacuous in the 
deterministic or near-deterministic case 
(where $\sigma^2\approx 0$), and is weaker than our bounds. With a 
different choice of the step size, 
it is possible to get a non-vacuous bound even if $\sigma^2\rightarrow 0$, but 
the dependence on $\tau$ becomes even stronger.  
\cite{feyzmahdavian2016asynchronous} improve the 
bound to
\[
\Ocal\left(\frac{\sigma}{\sqrt{k}}+\frac{\tau^2+1}{k}\right)~~~\text{and}~~~
\Ocal\left(\frac{\sigma^2}{k}+\frac{\tau^4+1}{k^2}\right).
\]
in the convex and strongly convex case respectively. 
Even if $\sigma^2=0$, the iteration complexity is $\Ocal(\tau^2/\epsilon)$ and 
$\Ocal(\tau^2/\sqrt{\epsilon})$, and implies a quadratic dependence on $\tau$ 
(whereas in our bounds the scaling is linear). When $\sigma^2$ is positive, the 
effect of delay on the bound is negligible only up to $\tau=\Ocal(\sqrt[4](k))$ 
(in contrast to $\tilde{\Ocal}(\sqrt{k})$ or even $\tilde{\Ocal}(k)$ in our 
bounds). We note that there are several other works which 
study a similar setting (such as \cite{sra2015adadelay}), but do not result in 
bounds which improve on the above. 

Finally, we note that \cite{langford2009slow} attempt to show that for 
stochastic gradient descent with delayed updates, the dependence on the delay 
$\tau$ is negligible after sufficiently many iterations. 
Unfortunately, as pointed out in \cite{agarwal2011distributed}, the analysis 
contains a bug which make the results invalid.

\section{Framework and the Generating Functions Approach} 
\label{section:generating_functions_and_framework}

Throughout, we will assume that $F$ is a convex quadratic function specified 
by 
\begin{align} \label{quadratic_problem}
F(\bw) \coloneqq \frac{1}{2}\bw^\top A\bw + \bb^\top\bw +c,
\end{align}
where $A\in\reals^{d\times d}$ is a positive semi-definite matrix whose 
eigenvalues $a_1,\dots,a_d$ are in $[0,\mu]$ (where $\mu$ is the smoothness 
parameter), $\bb\in \reals^d$ and $c\in \reals$. To make the optimization 
problem meaningful, we further assume that $F$ is 
bounded from below, which implies that it has some minimizer $\bw^*\in\reals^d$ 
at which the gradient vanishes (for completeness, we provide a proof in  
\lemref{lem:bounded_quadratic} in the 
supplementary material). 
Letting $\be_k = \bw_k-\bw^*$, it is easily verified that
\begin{align}\label{eq:convex_value_error}
F(\bw_k)-F(\bw^*) = \frac{1}{2}\left\| \sqrt{A} (\bw-\bw^*)\right\|^2 = 
\frac{1}{2}\left\| \sqrt{A} \be_k\right\|^2,
\end{align}
so our goal will be to analyze the dynamics of $\be_k$. 

To explain our technique, consider the iterates of DGD on the function $F$, 
which can be written as
$\bw_{k+1}=\bw_k-\eta \nabla F(\bw_{k-\tau}) =\bw_k-\eta (A\bw_{k-\tau} +\bb)$. 
Since $\nabla F(\bw^*)=0$, we have $\bw^*=\bw^*-\eta (A\bw^* +\bb)$, by which 
it follows that the error term $\be_k = \bw_k-\bw^*$, satisfies the recursion 
$\be_{k+1} = \be_k -\eta A  \be_{k-\tau}$, and (by definition of the algorithm) 
$\be_0=\be_1=\ldots=\be_\tau$. By some simple arguments, our analysis then 
boils down to bounding the elements of the scalar-valued version of this 
sequence, namely
\begin{align} \label{basic_dynamics}
\begin{aligned}
b_0&=\dots=b_\tau\in\reals,\\
b_{k+1} &= b_k - \alpha b_{k-\tau},~k\ge \tau,
\end{aligned}
\end{align} 
for some integer $\tau\ge0$ and non-negative real number $\alpha\ge0$. To 
analyze this sequence, we rely on tools from the area of generating functions, 
which have proven very effective in studying growth rates of 
sequences in many areas of mathematics. We now turn to briefly describe these 
functions and our approach (for general surveys on generating functions, see 
\cite{wilf2005generatingfunctionology,flajolet2009analytic,stanley1986enumerative},
to name a few).


Generally speaking, generating functions are formal power series associated 
with infinite sequences of numbers
. Concretely, given a sequence $(b_k)$ of numbers in a ring $R$, we define the 
corresponding generating function as a formal power series in $z$, defined as 
$
f(z) = \sum_{k=0}^\infty b_k z^k
$.
The set of all formal power series in $z$ over $R$ is denoted by $R[[z]]$. 
Moreover, given two power series defined by sequences $(a_k)$ and $(c_k)$, we 
can define their addition as the power series corresponding to $(a_k+c_k)$, and 
their multiplication as the coefficients of the Cauchy product of the 
power series, namely $(\sum_k a_k z^k)(\sum_k c_k z^k) = \sum_k 
(\sum_{l=0}^{k}a_l c_{k-l}) z^k$. In particular, over the reals, $\reals[[z]]$ 
endowed with 
addition and multiplication is a commutative 
ring, and the set of matrices with elements in $\reals[[z]]$ (with the standard 
addition and 
multiplication operations) forms a matrix algebra, denoted by 
$\mathcal{M}(\reals[[z]])$. 
We will often use the fact that any matrix, whose 
entries are power series with scalar coefficients, can also be written as a 
power series with matrix-valued coefficients: More formally, 
$\mathcal{M}(\reals[[z]])$ is naturally identified with the ring of formal 
power series with real matrix coefficients $\mathcal{M}(\reals)[[z]]$. To 
extract the coefficients of a given $M(z)\in \mathcal{M}(R[[z]])$, we shall use 
the conventional bracket notation $[z^k]M(z)$, defined to be a matrix whose 
entries are the $k$'th coefficients of the respective formal power series. 

Returning to \eqref{basic_dynamics}, we write $(b_k)$ as a formal 
power series denoted by $f(z)$, and proceed as follows,
\begin{align} 
f(z) &= 
\sum_{k=0}^{\tau} b_k z^k + \sum_{k=\tau+1}^\infty (b_{k-1} - \alpha b_{k-\tau-1}) z^k
= \sum_{k=0}^{\tau} b_k z^k + \sum_{k=\tau+1}^\infty b_{k-1}  z^k -\alpha 
\sum_{k=\tau+1}^\infty  b_{k-\tau-1} z^k\notag\\
&= \sum_{k=0}^{\tau} b_k z^k +z\left(  f(z) - \sum_{k=0}^{\tau-1} b_{k}z^k\right)   -\alpha z^{\tau+1}f(z)
 = b_0 +(z-\alpha z^{\tau+1})f(z)~.\label{eq:generating_function_derivation}
\end{align}
Denoting 
\[\pi_\alpha(z) ~\coloneqq~ 1-z+\alpha z^{\tau+1}
\] and rearranging terms gives 
\begin{align} \label{eq:coefficients_of_b}
f(z) = \frac{ b_0 }{\pi_\alpha(z)}\quad \implies\quad
b_k = [z^k] f(z) =  [z^{k}] \frac{b_0}{\pi_\alpha(z)}~
\end{align}
(by a 
well-known fact, $\pi_\alpha(z)$ is invertible in $\reals[[z]]$, as its constant term 1 is trivially invertible in $\reals$ -- see surveys 
mentioned above). We now see that the problem of bounding the coefficients $(b_k)$ is reduced 
to that of estimating the coefficients of the rational function 
${1}/{\pi_\alpha(z)}$, written as a power series.
Note that for the analogous problem where the elements of the sequence are 
vectors $(\bb_k)_{k=0}^\infty$ and the factor $\alpha$ is replaced by $\alpha  
A$ for some square matrix $A$, the same derivation as above yields 
$\sum_{k=0}^\infty \bb_k z^k = (I-z+\alpha A z)^{-1}\bb_0$ (likewise, $I-z+ A z$ is invertible in $\mathcal{M}(\reals)[[z]]$
as its constant term $I$ is invertible in $\mathcal{M}(\reals)$).

To estimate the coefficients of ${1}/{\pi_\alpha(z)}$, we form its 
corresponding partial fraction decomposition. First, we note that as a 
polynomial of degree $\tau+1$, $\pi_{\alpha}(z)$ has $\tau+1$ roots 
$\zeta_1,\ldots,\zeta_{\tau+1}$ (possibly 
complex-valued, and all non-zero since $\pi_\alpha(0)=1$ for any 
$\alpha\in\reals$). Assuming $\alpha$ is chosen so that all the roots are 
distinct (equivalently, 
$\pi'_\alpha(\zeta_i)\neq0,$ for 
$i\in[\tau+1]$), we have by a standard derivation
\begin{align*}
\frac{1}{\pi_\alpha(z)} = \sum_{i=1}^{\tau+1} \frac{1}{\pi_\alpha'(\zeta_i) (z-\zeta_i) }
= \sum_{i=1}^{\tau+1} \frac{-1}{\pi_\alpha'(\zeta_i) \zeta_i }\cdot \frac{1}{1- \frac{z}{\zeta_i} }
= \sum_{i=1}^{\tau+1} \frac{-1}{\pi_\alpha'(\zeta_i) \zeta_i} \sum_{k=0}^\infty \left(\frac{z}{\zeta_i}\right)\!^k.
\end{align*}
Thus,
\begin{align}\label{z_coefficient}
[z^k] \left(\frac{1}{\pi_\alpha(z)}\right) = \sum_{i=1}^{\tau+1} 
\frac{-1}{\pi_\alpha'(\zeta_i) \zeta_i^{k+1}}~.
\end{align}
To bound the magnitude of $1/\zeta_i$ and $\pi'_\alpha(\zeta_i)$, we invoke the 
following lemma, whose proof (in the supplementary material) relies on some 
tools from complex analysis: 
\begin{lemma} \label{lem:roots_bound}
	Let $\alpha\in\left(0,{1}/{20 (\tau+1)}\right]$, and assume $|\zeta_1|\le|\zeta_2|\le\dots\le|\zeta_{\tau+1}|$, then 
	\begin{enumerate}
		\item		$\zeta_1$ is a real scalar satisfying 
		$1/\zeta_1 \le 1- \alpha$, and for $i>1$, $|1/\zeta_i|\le 
		1-\frac{3}{2(\tau+1)}$. 
		\item  $|\pi'_\alpha(\zeta_i)|>1/2$, for any $i\in[\tau+1]$.
	\end{enumerate} 
\end{lemma}
With this lemma at hand, we have
\begin{align*} 
\begin{aligned}
\left|[z^k] \left(\frac{1}{\pi_\alpha(z)}\right)\right|
&\le 
2(1-\alpha)^{k+1} + 2\tau\left(1-\frac{3}{2(\tau+1)}\right)^{k+1}
\le 2(1-\alpha)^{k+1} \left(1+\tau\exp\left(-\frac{k+1}{\tau+1}\right)\right)~,
\end{aligned}
\end{align*}
where the last inequality is due to \lemref{lem:k_regime} (provided in 
the supplementary material). 
Moreover, one can use elementary arguments to show that $|[z^k]1/\pi_\alpha(z)|\le 1$ for any $k\ge0$, as long as $\alpha\in [0,1/\tau]$  (see \lemref{lem:bound_1} in the supplementary 
material). Overall, for any $\tau\ge0$, we have 
\begin{align} \label{ineq:coefficients_bound}
\begin{cases}
\left|[z^k] \left(\frac{1}{\pi_\alpha(z)}\right)\right|
 \le 1 & 0\le k \le (\tau+1)\ln(2(\tau+1))-1,\\
\left|[z^k] \left(\frac{1}{\pi_\alpha(z)}\right)\right|
 \le 3(1-\alpha)^{k+1}& k \ge (\tau+1)\ln(2(\tau+1)),
\end{cases}
\end{align}
which, using \eqref{eq:coefficients_of_b}, gives the desired bounds on the elements $(b_k)$ defined in 
\eqref{basic_dynamics}.


\section{Deterministic Delayed Gradient Descent}
\label{section:deterministic_delayed}
We start by analyzing the convergence of DGD for $\lambda$-strongly convex and 
$\mu$-smooth quadratic functions, where the eigenvalues of $A$ are assumed to 
lie in $[\lambda,\mu]$ for some $\mu\geq \lambda>0$.

Following the same line of the derivation as in 
\eqref{eq:generating_function_derivation}, we obtain 
$\mathbf{e}(z) = (I-Iz+\eta A z^{\tau+1})^{-1} \be_0$.
Letting $[d]:=\{1,2,\ldots,d\}$, it follows that for any $k\ge 
(\tau+1)\ln(2(\tau+1))$,
\begin{align} \label{ineq:generating_bound}
\begin{aligned}
\|\be_k\|&=\|[z^k]\left((I-Iz+\eta A z^{\tau+1})^{-1}\be_0\right)\|
\stackrel{(a)}{=}\|[z^k]\left((I-Iz+\eta A z^{\tau+1})^{-1}\right)\be_0\|\\
&\stackrel{(b)}{\le} \max_{i\in[d]}\left|[z^k] \frac{1}{\pi_{\eta a_i}(z)} 
\right|\|\be_0\| \stackrel{(c)}{\le} 3 \max_{i\in[d]} (1-\eta 
a_i)^{k+1}\|\be_0\| \stackrel{(d)}{\le} 3(1-\eta \lambda)^{k+1}\norm{\be_0}~,
\end{aligned}
\end{align}
where $(a)$ follows by the linearity, $(b)$ by eigendecomposition of $A$ (that 
reveals that the spectral norm of a matrix polynomial equals the 
absolute value of the same polynomial in one of its eigenvalues), 
$(c)$ by 
Ineq.  
\ref{ineq:coefficients_bound} for $\eta
\mu\in\left(0,{1}/{(20(\tau+1))}\right]$, 
and $(d)$ by the fact that $a_i\geq \lambda$ for all $i$. Moreover, by 
\eqref{eq:convex_value_error} and the fact that all eigenvalues of $A$ are at 
most $\mu$, we arrive at the following bound:
\begin{theorem}\label{thm:convergence}
	For any delay $\tau\ge0$ and $k\ge (\tau+1)\ln(2(\tau+1))$, running DGD 
	with step 
	size $\eta\in\left(0,{1}/{(20\mu(\tau+1))}\right]$ on a 
	$\mu$-smooth, $\lambda$-strongly convex quadratic function
	yields
	\begin{align*}
	F(\bw_k)-F(\bw^*) 
	&\le 5\mu\left(1-\eta\lambda\right)^{2(k+1)}\|\bw_0 - \bw^*\|^2.
	\end{align*}
In particular, setting $\eta=\Omega(1/\mu\tau)$, we get 
that
\[
F(\bw_k)-F(\bw^*) ~\leq~ 
5\mu\norm{\bw_0-\bw^*}^2\exp\left(-\Omega\left(\frac{k\lambda}{\mu\tau}\right)\right)~.
\]
\end{theorem}
Note that the assumption that $k\geq 
(\tau+1)\ln(2(\tau+1))$ is very mild, since if $k\leq \tau$ then the algorithm trivially makes no updates after $k$ rounds.

We now turn to analyze the case of $\mu$-smooth convex quadratic functions, 
where the eigenvalues of the matrix $A$ are assumed 
to lie in $[0,\mu]$. Following the same derivation as in Ineq. 
\ref{ineq:generating_bound} and using Ineq. \ref{eq:convex_value_error}, we 
have for any $k\ge (\tau+1)\ln(2(\tau+1))$ and 
$\eta\in\left(0,{1}/{(20\mu(\tau+1))}\right]$,
\begin{align} 
\begin{aligned}
F(\bw_k)-F(\bw^*) &= \frac{1}{2}\|\sqrt{A}\be_k\|^2 =
 \frac{1}{2}\|\sqrt{A}[z^k]\left((I-Iz+\eta Az^{\tau+1})^{-1}\right)\be_0\|^2\\
&\stackrel{(a)}\le \frac{1}{2}\left( 3\max_{i\in[d]} \sqrt{a_i}(1-\eta 
a_i)^{k+1}\right)^2\|\be_0\|^2
\stackrel{(b)}\le 
\frac{9}{4e\eta (k+1)}\|\be_0\|^2~,
\end{aligned}
\end{align}
where $e= 2.718...$ is Euler's number, $(a)$ is by the 
fact that the spectral norm of a matrix polynomial equals the absolute value of 
the same polynomial in one of its eigenvalues, and $(b)$ is by 
the fact that $\sqrt{a_i}(1-\eta a_i)^{k+1}\le {1}/{\sqrt{2e \eta (k+1)}}$ for 
any $i\in[d]$ (see \lemref{lem:convex_power_bound} in the supplementary 
material).
We have thus arrived at the following bound for the convex case:
\begin{theorem}\label{thm:convex_convergence}
	For any delay $\tau\ge0$ and $k\ge (\tau+1)\ln(2(\tau+1))$, running 
	DGD with step size $\eta\in\left(0,{1}/{(20\mu(\tau+1))}\right]$ on a 
	$\mu$-smooth convex quadratic function yields
	\begin{align*}
	F(\bw_k)-F(\bw^*) &\le 
	\frac{9}{4e\eta(k+1)}\norm{\bw_0-\bw^*}^2~.
	\end{align*}
In particular, if we set $\eta=\Omega(1/\mu\tau)$, we get that
\[
F(\bw_k)-F(\bw^*)\leq \Ocal\left(\frac{\mu\tau \norm{\bw_0-\bw^*}^2}{k}\right)~.
\]
\end{theorem}

As discussed in the introduction, the theorems above imply that a delay of 
$\tau$ increases the iteration complexity by a factor of $\tau$. We now show 
lower bounds which imply that this linear dependence on $\tau$ is  
unavoidable, for a large family of gradient-based algorithms (of which 
gradient descent is just a special case). Specifically, we will consider any 
iterative algorithm producing iterates $\bw_0,\bw_1,\ldots$ which satisfies the 
following:
\begin{equation}\label{eq:assumplowbound}
\bw_0=\ldots=\bw_{\tau}=\mathbf{0}~~~\text{and}~~~\forall k\geq t,~
\bw_{k+1} \in  \text{span}\{ \nabla F(\bw_{0}), \nabla F(\bw_{1}),\dots, \nabla 
F(\bw_{k-\tau}) \}~.
\end{equation}
This is a standard assumption in proving optimization lower bounds (see 
\cite{nesterov2004introductory}), and is 
satisfied by most standard gradient-based methods, and in particular our DGD 
algorithm. We also note that this algorithmic assumption can be relaxed at the 
cost of a more involved proof, similar to 
\cite{nemirovskyproblem,woodworth2016tight} in the non-delayed case.

\begin{theorem}\label{thm:lower_bound}
	Consider any algorithm satisfying \eqref{eq:assumplowbound}. Then the 
	following holds for any $k\ge \tau+1$ and sufficiently large dimensionality 
	$d$:
	\begin{itemize}[leftmargin=*]
		\item There exists a 
		$\mu$-smooth, 
		$\lambda$-strongly convex function $F$ over $\reals^d$, such that
		\begin{align*}
		F(\bw_k) - F(\bw^*) ~\ge~ \frac\lambda4\exp\left( 
		-\frac{5k}{\left(\sqrt{\mu/\lambda}-1\right)(\tau+1)}\right){\|\bw_0-\bw^*\|^2}~.
		\end{align*}
		\item There exists a $\mu$-smooth, convex quadratic function $F$ over 
		$\reals^d$, such that
		\begin{align*}
		F(\bw_k) - F(\bw^*)~\ge~ 
		\frac{\mu(\tau+1)^2\norm{\bw_0 - \bw^*}^2}{45k^2}~.
		\end{align*}
	\end{itemize}
\end{theorem}
The proof of the theorem is very similar to standard optimization lower bounds 
for gradient-based methods without delays (e.g. 
\cite{nesterov2004introductory,lan2015optimal}), 
and is presented in the supplementary material. In fact, our main contribution 
is to recognize that the proof technique easily extends to incorporate delays.

In terms of iteration complexity, these bounds correspond to 
$\Omega\left(\tau\cdot \sqrt{\mu/\lambda}\cdot 
\ln\left(\lambda\norm{\bw_0-\bw^*}^2/\epsilon\right)
\right)$ in the strongly convex case, and $\Omega\left(\tau\cdot 
\sqrt{\mu\norm{\bw_0-\bw^*}^2/\epsilon^2}\right)$ in the convex case, which 
show that the linear dependence on $\tau$ is inevitable. The dependence on the 
other problem parameters is somewhat better than in our upper bounds, but this 
is not just an artifact of the analysis: In our delayed setting, the lower 
bounds can be matched by running \emph{accelerated} gradient descent (AGD)
\cite{nesterov2004introductory}, 
where each time we perform an accelerated gradient descent step, and then stay 
idle for $\tau$ iterations till we get the gradient of the current point. 
Overall, we perform $k/\tau$ accelerated gradient steps, and can apply the 
standard analysis of AGD to get an iteration complexity which is $\tau$ times 
the iteration complexity of AGD without delays. These match the lower bounds 
above up to constants. We believe it is possible to prove a similar upper bound 
for AGD performing an update with a delayed gradient at every iteration (like 
our DGD procedure), but the analysis is more challenging than for plain 
gradient descent, and we leave it to future work.

%


\section{Stochastic Delayed Gradient Descent}
\label{section:stochstic_delayed}

In this section, we study the case of noisy (stochastic) gradient updates, and 
the SDGD algorithm, in which the influence of the delay is quite different than 
in the noiseless case. Instantiating SDGD for quadratic $F(\bw)$ (defined in 
(\ref{quadratic_problem})) results in the following update rule
\begin{align} \label{stochastic_dynamics}
\bw_{k+1}=\bw_k-\eta \nabla F(\bw_{k-\tau} + \epsilon_k) =\bw_k-\eta 
(A\bw_{k-\tau} +\bb + \bxi_k)~,
\end{align}
where $\bxi_{k},~k\ge0$ are independent zero-mean noise terms satisfying 
$\E[\|\bxi_t\|^2]\le\sigma^2$. As before, in terms of the error term 
$\be_k=\bw_k-\bw^*$, \eqref{stochastic_dynamics} reads as $\be_{k+1} 
=\be_k-\eta A\be_{k-\tau} -\eta \bxi_k$. Given a realization of $(\bxi_{k})$, 
we denote its associated formal power series by $g(z)\coloneqq 
\sum_{k=\tau}^{\infty} \bxi_k z^k$. By an analysis similar to before, we get 
that the formal power series of the error terms $(\be_k)$ satisfies
\begin{align*}
\be(z) = (I-Iz + \eta A z^{\tau+1})^{-1}(\be_0 -\eta g(z))~.
\end{align*}
We can now bound the error terms by extracting the corresponding coefficients 
of $\be(z)$. Letting $D\coloneqq (I-Iz + \eta A z^{\tau+1})^{-1}$,
we have for any $k\ge(\tau+1)\ln(2(\tau+1))$
\begin{align}
2\cdot\E[F(\bw_k)-F(\bw^*)]&=\E\left[\left\|\sqrt{A}\be_k\right\|^2\right] 
= \E\left[\left\|\sqrt{A}~[z^k] \left(D(\be_0 -\eta g(z)) 
\right)\right\|^2\right]\notag\\
&\stackrel{(a)}{=} 
\left\|\sqrt{A}~[z^k] D\be_0\right\|^2 + \eta^2 
\E\left[\left\|\sqrt{A}~[z^k]\left(D g(z) 
\right)\right\|^2\right]\notag\\
&\stackrel{(b)}{=} 
\|\sqrt{A}~[z^k] D\be_0\|^2  + \eta^2\E\left[\|\sqrt{A}\sum_{i=0}^k 
\left([z^i]D  
\right)\bxi_{k-i}\|^2\right]\notag\\
&\stackrel{(c)}{\le} 
\|\sqrt{A}~[z^k] D\|^2\|\be_0\|^2  + \eta^2\sigma^2\sum_{i=0}^k 
\left\|\sqrt{A}~[z^i]D  \right\|^{2},
\label{ineq:stochastic}
\end{align}
where $(a)$ follows by the linearity of the bracket operation $[z^k]$ and the 
assumption that $\E[\bxi_k]=0$ for all $k$ (hence $\E[g(z)]=0$), $(b)$ follows 
by the Cauchy product for formal power series,  and $(c)$ by the hypothesis 
that $\bxi_k$ are independent and satisfy $\E[\|\bxi_k\|^2]\le\sigma^2$ for all 
$k$. 
We then upper bound both terms, building on Ineq. \ref{ineq:coefficients_bound} 
(see the supplementary material for a full derivation), resulting in the 
following theorem:
\begin{theorem} \label{thm:stochastic_rate}
Assuming the step $\eta$ satisfies $\eta\in 
(0,\frac{1}{20\mu(\tau+1)}]$, and $k\geq (\tau+1)\ln(2(\tau+1))$, 
the 
following holds for SDGD:
\begin{itemize}[leftmargin=*]
	\item  For $\lambda$-strongly convex, $\mu$-smooth quadratic convex 
	functions, $\E[F(\bw_k)-F(\bw^*)]$ is at most
	\[
	5 \mu \exp(-2\eta \lambda(k+1))\norm{\bw_0-\bw^*}^2
	~+~\frac{\eta^2\sigma^2}{2}\left(\mu(\tau+1)\ln(2(\tau+1))
	+\frac{1+e+\ln(\frac{1}{\eta\lambda})}{e\eta}\right)~.
	\]
	In particular, by tuning $\eta$ appropriately,
	\[
	\E\left(F(\bw_k)-F(\bw^*)\right)
	~\leq~\tilde{\Ocal}\left(\frac{\sigma^2}{\lambda 
	k}+\mu\norm{\be_0}^2\exp\left(-\frac{\lambda k
	}{10\mu\tau}\right)\right)~.
	\]	
	\item For $\mu$-smooth quadratic convex functions, $\E[F(\bw_k)-F(\bw^*)]$ 
	is at most
	\[
\frac{9\norm{\bw_0-\bw^*}^2}{4e\eta(k+1)}
~+~ 
\eta^2\sigma^2\left(\mu(\tau+1)\ln(2(\tau+1))
+\frac{9}{2e\eta }(1+\ln(k+1))\right)~.
	\]
	In particular, by tuning $\eta$ appropriately,
	\[
	\E\left(F(\bw_k)-F(\bw^*)\right)
	~\leq~
	\tilde{\Ocal}\left(\frac{\norm{\bw_0-\bw^*}\sigma}{\sqrt{k}}
	+\frac{\norm{\bw_0-\bw^*}^2\mu\tau}{
		k}\right)~.
	\]
\end{itemize}
\end{theorem}
As discussed in the introduction in detail, the theorem implies that the effect 
of $\tau$ is negligible once $k$ is sufficiently large.


\bibliographystyle{plain}
\bibliography{bib}

\newpage

\appendix
\section{Proof of \lemref{lem:roots_bound}}
\label{app:section:roots}

	Recall that $\pi_\alpha(z) = 1-z+\alpha z^{\tau+1}$, and its roots, denoted by $\zeta_i$, are ordered such that $|\zeta_1|\le|\zeta_2|\le\dots\le|\zeta_{\tau+1}|$.
	In order to bound from above the magnitude of $1/\zeta_i$,  we analyze a related polynomial $p_\alpha(z) = z^{\tau+1}\pi_a(1/z)$ which takes the following explicit form
	\[
	p_\alpha(z) ~=~ z^{\tau+1}-z^{\tau}+\alpha~=~ (z-1)z^{\tau}+\alpha.
	\]
	The roots of $p_\alpha$ are precisely $1/\zeta_i$ (note that, $\pi_a(0) = 
	1\neq 0 $, hence $\zeta_i\neq0,~i\in\{1,\ldots,\tau+1\}$). Thus, bounding 
	from above (below) the magnitude of the roots of $p_\alpha(z)$ gives an 
	upper (lower) bound for $|1/\zeta_i|$.

	We first establish that for any $\alpha \in \left(0,\frac{1}{20(\tau+1)}\right]$, $p_\alpha$ has a real-valued root in $\left(1-\frac{1}{2(\tau+1)},1-\alpha\right]$. Indeed, for any such $\alpha$, we have on the one hand,
	\[
	p_{\alpha}(1-\alpha) ~=~ -\alpha(1-\alpha)^{\tau}+\alpha~=~ \alpha\left(1-(1-\alpha)^{\tau}\right)~\ge~0,
	\]
	and on the other hand (using the fact that $(1-1/2x)^x \ge 1/2$ for all $x\ge1$),
	\begin{align}
	p_\alpha\left(1-\frac{1}{2(\tau+1)}\right) &= 
	-\frac{1}{2(\tau+1)}\left(1-\frac{1}{2(\tau+1)}\right)^\tau + \alpha 
	\notag\\
	&= 
	-\frac{1}{2(\tau+1)}\left(\left(1-\frac{1}{2(\tau+1)}\right)^{\tau+1}\right)^{\frac{\tau}{\tau+1}}
	 + \alpha 
	\notag\\&\le 
	-\frac{1}{2(\tau+1)}\left(\frac{1}{2}\right)^{\frac{\tau}{\tau+1}} + \alpha 
	< -\frac{1}{20(\tau+1)}+\alpha \le0, 
	\label{ineq:lower_bound_dominant}
	\end{align}
	so by continuity of $p_z$, we get that a real-valued root exists in 
	$\left(1-\frac{1}{2(\tau+1)},1-\alpha\right]$.

	Next, we show that $\tau$ non-dominant roots of $p_\alpha$ are 
	of absolute value smaller than $R= 1-\frac{3}{2(\tau+1)}$. To this end, we invoke Rouch\'{e}'s theorem, which states 
	that for any two holomorphic functions $f,g$ in some region $K\subseteq 
	\mathbb{C}$ with closed contour $\partial K$, if $|g(z)|<|f(z)|$ for any $z\in 
	\partial K$, then $f$ and $f+g$ have the same number of zeros (counted with 
	multiplicity) inside $K$. In 
	particular, choosing $f(z)=-z^\tau$, $g(z)=z^{\tau+1}+\alpha$ and $K=\{z:|z|\le R\}$, it follows that if $|z^{\tau+1}+\alpha|<|-z^\tau|$ for all $z$ 
	such that $|z|=R$, then $f+g$ (which equals our polynomial 
	$p_\alpha$) has the same number of zeros as $f=-z^\tau$ inside $K$ (namely, exactly 
	$\tau$). However, since $p_\alpha$ is a degree $\tau+1$ polynomial, 
	it has exactly $\tau+1$ roots, so the only root of absolute value larger 
	than $R$ is the real-valued one we found earlier. It remains to verify the condition $|z^{\tau+1}+\alpha|<|-z^\tau|$ for all $z$ 
	such that $|z|=R$. For that, it is sufficient to show that 
	$|z^{\tau+1}|+\alpha < |z^{\tau}|$ for all such $z$, or equivalently, 
	$R^\tau> \alpha+R^{\tau+1}$. 
	\begin{align*}
	R^{\tau}&~=~\left(1-\frac{3}{2(\tau+1)}\right)^{\tau} ~=~ 
	\left(1-\frac{3}{2(\tau+1)}\right)^{\tau}
	-\left(1-\frac{3}{2(\tau+1)}\right)^{\tau+1}+\left(1-\frac{3}{2(\tau+1)}\right)^{\tau+1}\\
	&~=~ \frac{3}{2(\tau+1)}\left(1-\frac{3}{2(\tau+1)}\right)^{\tau}+R^{\tau+1}.
	\end{align*}
	By the inequality $1-1/(x+1)\ge\exp(-1/x)$ (see \lemref{lem:tech1} below), we have
	\begin{align*}
	1-\frac{3}{2(\tau+1)} 
	\ge
	\exp\left(\frac{-1}{2/3\tau-1/3}\right) \quad\implies\quad 
	\left(1-\frac{3}{2(\tau+1)}\right)^{\tau} 
	\ge
	\exp\left(\frac{-\tau}{2/3\tau-1/3}\right) 
	\end{align*}
	It is straightforward to verify that 
	\begin{align*}
	\frac{-\tau}{2/3\tau-1/3}\ge -3,~\tau\ge1,
	\end{align*}
	implying that
	\begin{align*}
	R^{\tau}&~=~ \frac{3}{2(\tau+1)}\left(1-\frac{3}{2(\tau+1)}\right)^{\tau}+R^{\tau+1} \\
	&~\ge~ \frac{3}{2e^3(\tau+1)}+R^{\tau+1}\\
	&~>~ \frac{1}{20(\tau+1)}+R^{\tau+1}~\ge~\alpha+R^{\tau+1}~,
	\end{align*}
	where in the last inequality we used the assumption that $\alpha\in 
	\left(0,\frac{1}{20(\tau+1)}\right]$. As mentioned earlier, the roots of $p_\alpha$ are exactly the reciprocals of the roots of $\pi_\alpha$, therefore we conclude
	\begin{align} \label{ineq:bound_on_non_dominant_roots}
	\left|\frac{1}{\zeta_i}\right|\le 1-\frac{3}{2(\tau+1)},~ i\in[\tau].
	\end{align}
	We now turn to bound $|\pi'_\alpha(\zeta_i)|$ from above. By definition, 
	any root of $\pi_a$ satisfies $\alpha\zeta_i^{\tau+1}-\zeta_i + 1 =0$. 
	Thus, $\alpha\zeta_i^{\tau} =\frac{\zeta_i-1}{\zeta_i}$ (note that as 
	mentioned in the first part of the proof, $\zeta_i\neq0$). This, in turn, 
	gives 
	\begin{align} \label{eq:roots_in_derivarive}
	\pi'_\alpha(\zeta_i) &= \alpha(\tau+1)\zeta_i^\tau - 1 = \frac{(\tau+1)(\zeta_i-1)}{\zeta_i} - 1 \nonumber
	\\&= \frac{(\tau+1)(\zeta_i-1) - \zeta_i}{\zeta_i} \nonumber
	\\&= \frac{(\tau+1)\zeta_i-(\tau+1) - \zeta_i}{\zeta_i} \nonumber
	\\&= \frac{\tau\zeta_i-(\tau+1)}{\zeta_i} 
	= \tau - \frac{(\tau+1)}{\zeta_i} 
	= (\tau+1)\left( \frac{\tau}{\tau+1} - \frac{1}{\zeta_i} \right).
	\end{align}
	In the previous parts of the proof, we showed that the distance from any 
	root of $p_\alpha$ to the contour $\{z~|~|z|=1-1/(\tau+1)\}$ is bounded 
	from below by $\frac{1}{2(\tau+1)}$ (Ineq. \ref{ineq:lower_bound_dominant} 
	and Ineq. \ref{ineq:bound_on_non_dominant_roots}), therefore 
	\begin{align*}
	|\pi'_\alpha(\zeta_i)| = (\tau+1)\left| 1- \frac{1}{\tau+1} - 
	\frac{1}{\zeta_i} \right| \ge 
	\frac{\tau+1}{2(\tau+1)}=\frac{1}{2},~i=1,\dots,\tau+1~,
	\end{align*}
	thus concluding the proof.

\section{Technical Lemmas } \label{app:section:technical_lemmas}

\begin{lemma} \label{lem:bound_1}
For any $\alpha\in[0,1/\tau]$ and $k\ge0$, it holds that $|[z^k]1/\pi_\alpha(z)|\le1$.
\end{lemma}
\begin{proof}
Recall that by \eqref{eq:coefficients_of_b}, $b_k  =  [z^{k}] \frac{b_0}{\pi_\alpha(z)}$. Therefore, suffices it to prove that $(b_k)$ (defined in \ref{basic_dynamics}) with $b_0=1$ and  $\alpha\in[0,1/\tau]$, satisfies $|b_k|\le1$ for any $k\ge0$.

For the sake of simplicity, we slightly extend $(b_k)$ to the negative indices by defining $b_{-\tau}=b_{-\tau+1}=\dots =b_{-1}=1$. We proceed by full induction. The base case holds trivially by the definition of the initial conditions of $b_k$. For the induction step, suppose that $|b_0|,\dots,|b_k|\le1$. We have $b_{k+1}=b_{k}-\alpha b_{k-\tau}$, and therefore
	\begin{align*}
	b_{k+1}=(1-\alpha)b_k+\alpha(b_k-b_{k-\tau}) = (1-\alpha)b_k+\alpha\sum_{i=k-\tau}^{k-1} (b_{i+1}-b_{i}).
	\end{align*}
	Using the recurrence relation again, this equals
	\begin{align*}
	(1-\alpha)b_k+\alpha\sum_{i=k-\tau}^{k-1} (-\alpha b_{i-\tau}) = (1-\alpha)b_k-\alpha\left(\alpha\sum_{i=k-2\tau}^{k-\tau-1} b_i\right).
	\end{align*}
	By the induction hypothesis, this equals $(1-\alpha)b_k+\alpha r_k$, where $|r_k|\leq \alpha\tau\leq 1$. Thus, $b_{k+1}$ is a weighted average of $b_k$ and $r_k$ which are both in $[-1,+1]$ by the induction hypothesis and the above, implying that we must have $b_{k+1} \in [-1,+1]$ as well. Thus, proving the induction step.
	
\end{proof}

\begin{lemma} \label{lem:bounded_quadratic}
Let $F(\bw) \coloneqq \frac{1}{2}\bw^\top A \bw + \bb^\top \bw,~A\in\reals^{d\times d}, \bb\in\reals^d$ be a convex quadratic function defined over $\reals^d$. If $F$ is bounded from below, then $F$ has a minimizer at which the gradient vanishes.
\end{lemma}
\begin{proof}
Since $F$ is convex and twice differentiable, $A$ is positive semidefinite. In 
particular, we have $\reals^d = \ker(A) \oplus \text{im}(A)$ (namely, the 
direct sum of the null space and the image space of $A$). Thus, $\bb$ can 
be expressed as a sum of two  orthogonal vectors $\bb= \bb^\perp + \bar{\bb}$, 
where $\bb^\perp\in\ker(A)$ and $\bar{\bb}\in \text{im}(A)$. For any 
$\alpha\in\reals$, we have
\begin{align*}
F(\alpha \bb^\perp) = \frac{1}{2}((\bb^\perp)^\top A \bb^\perp)\alpha^2 + \alpha \bb^\top\bb^\perp 
=  \alpha\|\bb^\perp\|^2.
\end{align*}
By the hypothesis, $F$ is bounded from below, hence  $\bb^\perp$ must vanish 
(otherwise we can take $\alpha\rightarrow -\infty$ and make $F$ as negative as 
we wish). In particular, $\bb=\bar{\bb}\in\text{im}(A)$. Let $\by\in\reals^d$ 
be such that $A\by=\bb$, then $\nabla F(-\by)=A(-\by)+\bb=0$. Lastly, $F$ is 
convex, therefore $-\by$ must be a (global) minimizer, thus concluding the 
proof.
\end{proof}

\begin{lemma} \label{lem:tech1}
For any $x>0$, it holds that	$1-1/(x+1)\ge\exp(-1/x)$.
\end{lemma}
\begin{proof}
Since $(\ln(1+x))'=1/(1+x)>0$ for any $x>-1$, it follows by the mean-value theorem that for any $x>0$ 
\begin{align*}
\ln(1+x) = \ln(1+x) -\ln(1) =\frac{1}{1+\xi} x,
\end{align*}
for some $\xi\in (0, x)$, hence $\ln(1+x)  \le x$ for any $x>0$. In particular, for any $x>0$ we have
\begin{align*}
\ln\left(1+\frac{1}{x}\right) \le \frac{1}{x}\implies \ln\left(\frac{x}{x+1}\right) \ge \frac{-1}{x}.
\end{align*}
Taking the exponent of both sides yields the desired lower bound.
\end{proof}

\begin{lemma} \label{lem:k_regime}
	Let $\tau\ge0$. If $\alpha\in(0,1/(20(\tau+1)]$ then
	\begin{align*}
	\left(\frac{1-\frac{3}{2(\tau+1)}}{1-\alpha}\right)^{k+1} &
	\le \exp\left(-\frac{k+1}{\tau+1}\right).
	\end{align*}
	In particular, for $k\ge(\tau+1)\ln(2(\tau+1))-1$, we have
	\begin{align} \label{k_regime_ineq}
	1+\tau\left(\frac{1-\frac{3}{2(\tau+1)}}{1-\alpha}\right)^{k+1} \le 3/2.
	\end{align}
\end{lemma}
\begin{proof}
	\begin{align*}
	\left(\frac{1-\frac{3}{2(\tau+1)}}{1-\alpha}\right)^{k+1} &= 
	\left(\frac{1-\alpha+\alpha - \frac{3}{2(\tau+1)}}{1-\alpha}\right)^{k+1}
	= \left(1 + \frac{\alpha - \frac{3}{2(\tau+1)}}{1-\alpha}\right)^{k+1}
	\le \left(1 + \alpha - \frac{3}{2(\tau+1)}\right)^{k+1},
	\end{align*}
	where the latter inequality follows from that fact that $\alpha < \frac{1}{20(\tau+1)}<\frac{3}{2(\tau+1)}$. Now,
	\begin{align*}
	\left(1 + \alpha - \frac{3}{2(\tau+1)}\right)^{k+1}&\le
	\exp\left((k+1)(\alpha - \frac{3}{2(\tau+1)})\right)
	\le \exp\left((k+1)\left(\frac{1}{20(\tau+1)}- \frac{3}{2(\tau+1)}\right)\right)\\
	&= \exp\left(-\frac{k+1}{\tau+1}\left(\frac{3}{2}-\frac{1}{20}\right)\right)
	\le \exp\left(-\frac{k+1}{\tau+1}\right).
	\end{align*}
	Lastly, to derive Ineq. \ref{k_regime_ineq}, we have
	\begin{align*}
	1+\tau\left(\frac{1-\frac{3}{2(\tau+1)}}{1-\alpha}\right)^{k+1} 
	&\le 1+(\tau+1)\exp\left(-\frac{k+1}{\tau+1}\right)
	= 1+\exp\left(\ln(\tau+1)-\frac{k+1}{\tau+1}\right)
	\le  1+1/2,
	\end{align*}
	where the last inequality by the assumption $k\ge(\tau+1)\ln(2(\tau+1))-1$. 
		
\end{proof}

\section{Proof of \thmref{thm:lower_bound}} 
\label{section:lower_bounds_deterministic}

The proof technique is based on a construction, first presented in 
\cite[Section 2.1.2]{nesterov2004introductory}, which has been proven effective 
in various settings of optimization since then.

First, we address the strongly convex case. Given  $\mu>\lambda>0$, we consider 
the following function (devised by \cite{lan2015optimal}):
\begin{align} \label{def:strongly_convex_construction}
F(\bw) \coloneqq \frac{\mu(\kappa-1)}{4}\left(\frac12\inner{A\bw, \bw}- 
\inner{\bepsilon_1,\bw}\right) + \frac{\lambda}{2}\norm{\bw}^2,
\end{align}
where $\kappa=\mu/\lambda$ as before, $\bepsilon_1$ denotes the first unit vector, and $A$ is a $d\times d$ matrix defined 
as follows
\begin{align} \label{def:lower_bound_hessian}
A = \left(\begin{matrix}
2& -1 & 0 & 0& \cdots & 0 & 0&0\\
-1& 2 & -1 & 0&\cdots & 0 & 0&0\\
\cdots&\cdots&\cdots&\cdots&\cdots&\cdots&\cdots&\cdots\\
0& 0& 0 &0& \cdots & -1 & 2&-1\\
0& 0& 0 &0& \cdots & 0 & -1 &\frac{\sqrt{\kappa}+1}{\sqrt{\kappa}+3}
\end{matrix}\right).
\end{align}
It can be easily verified that $F$ is $\mu$-smooth and $\lambda$-strongly 
convex function. Moreover, by \cite[Lemma 8]{lan2015optimal}, it follows that 
the minimizer of $f$ is $\bw^*=(q,q^2,\dots,q^d)$ where 
$q=(\sqrt{\kappa}-1)/(\sqrt{\kappa}+1)$. In particular, if $\bw\in\reals^d$ is 
a vector whose all non-zero entries are located in the first $m$ coordinates, 
where $m$ is such that $d\ge m/2 +{\log(1/2)}/{\log(q^2)}$, then 
\begin{align} \label{ineq:lower_bound_rate}
\frac{\|\bw\|^2}{\|\bw^*\|^2} \ge
\frac{ \sum_{i=m+1}^d q^{2i}}{\sum_{i=1}^d q^{2i} }
=q^{2(m+1)}\frac{ 1- q^{2(d-m-1)}}{1- q^{2d} }
\ge \frac12 q^{2(m+1)}\ge \frac12 
\exp\left(-\frac{4(m+1)}{\sqrt{\kappa}-1}\right),
\end{align}
where the last two inequalities follow from \cite[Lemma 9.b]{lan2015optimal} 
and \lemref{lem:tech1}, respectively. Therefore, by 
bookkeeping which entries of the iterates are non-zero, we can bound from below 
the distance to the minimizer. To this end, we will need the following lemma 
which, based on the tridiagonal structure of the Hessian of $F$, determines the 
non-zero entries:
\begin{lemma}\label{proof:lem:non_zero_entries}
Let $F:\reals^d\to\reals$ be a convex quadratic function specified as follows  $F(\bw)\coloneqq \frac{c}{2} \bw^\top A  \bw  + d \bepsilon_1^\top \bw$, where $A$ is a tridiagonal matrix and $c,d$ are real scalars. Assuming that the iterates produced by a given optimization algorithm satisfy $\bw_0=\dots=\bw_{\tau}=0$ and
	\begin{align*}
\forall k\ge\tau,~\bw_{k+1} \in  \text{span}\{ \nabla F(\bw_{0}), \nabla F(\bw_{1}),\dots, \nabla F(\bw_{k-\tau}) \},
	\end{align*}
	then $	\bw_k \in \text{span}\{\bepsilon_0,\bepsilon_1,\dots,\bepsilon_{\lfloor 
	k/(\tau+1)\rfloor } \}$ for all $k\ge0$ (where $\bepsilon_0$ denotes the 
	vector of all zeros, and $\bepsilon_i$ denote the $i$'th standard unit 
	vector).
\end{lemma}
\begin{proof}
	
	First, note that, given a vector $\bw\in\reals^d$, such that $\bw\in\text{span}\{
	\bepsilon_0, \bepsilon_1,\dots,\bepsilon_m \}$ for some $m\ge0$, we have
	\begin{align*}
	\nabla F(\bw) = cA\bw + d \bepsilon_1.
	\end{align*}
	Since the entries of $\bw$ are all zero start from the $m+1$ coordinate, $cA\bw $ is a linear combination of the first $m$ columns of $A$. Being a $A$ tridiagonal  matrix, it follows that all the entries of $cA\bw $ are zero, except for its first $m+1$ coordinates, that is, $cA\bw\in\text{span}\{\bepsilon_0,\bepsilon_1\dots, \bepsilon_{m+1}\}$. Together,  $\nabla F(\bw) = cA\bw + d \bepsilon_1 \in \text{span}\{\bepsilon_1\dots, \bepsilon_{m+1}\}$.
	
	We proceed by full induction. For $k=0,\dots,\tau$, the claim holds trivially. Now, assume the claim holds for all  $i\le k$, where $k\ge\tau$, we show that the claim holds for $k+1$. 
	By the induction hypothesis,  $	\bw_i \in \text{span}\{\bepsilon_0,\bepsilon_1,\dots,\bepsilon_{\lfloor i/(\tau+1)\rfloor } \}$	for all $i\le k$. Therefore, by the first part of the proof, we have,  $	\nabla F (\bw_i) \in \text{span}\{\bepsilon_1,\bepsilon_2,\dots,\bepsilon_{\lfloor i/(\tau+1)\rfloor +1 } \}$ for all $i\le k$, by which we conclude that 
	$	\text{span}\{\nabla F (\bw_0),\nabla F (\bw_1),\dots,\nabla F (\bw_{k-\tau})\} \subseteq \text{span}\{\bepsilon_1,\bepsilon_2,\dots,\bepsilon_{\lfloor (k-\tau)/(\tau+1)\rfloor +1 } \}$. Thus, by the linear span assumption, it follows that 
	\begin{align}
    \bw_{k+1} \in  \text{span}\{\bepsilon_1,\bepsilon_2,\dots,\bepsilon_{\lfloor (k-\tau)/(\tau+1)\rfloor +1 }\}.
	\end{align}
	Observing that,
	\begin{align*}
	\lfloor (k-\tau)/(\tau+1)\rfloor +1 = 	\lfloor (k-\tau)/(\tau+1)+1 \rfloor = \lfloor (k+1)/(\tau+1) \rfloor,
	\end{align*}
	concludes the proof.
\end{proof}

Overall, by \lemref{proof:lem:non_zero_entries}, the $k$'th iterate $w_k$, has all its entries zero, expect for (possibly) the first $\lfloor k/(\tau+1) \rfloor$ first coordinates. By Ineq. \ref{ineq:lower_bound_rate}, for any $\tau+1\le k \le 2\left(d- \frac{\log(1/2)}{2\log(q)}\right)$, we then have
\begin{align*}
\frac{\|\bw_k\|^2}{\|\bw^*\|^2} &\ge
\frac12 \exp\left(-\frac{4( \lfloor k/(\tau+1) \rfloor  +1)}{\sqrt{\kappa}-1}\right)
\ge
\frac12 \exp\left(-\frac{4(  k/(\tau+1)  +1)}{\sqrt{\kappa}-1}\right)\\
&\ge
\frac12 \exp\left(-\frac{5  k  }{(\sqrt{\kappa}-1)(\tau+1)}\right).
\end{align*}

For the convex case, we use a construction (devised by 
\cite{nesterov2004introductory}) similar to that of the strongly convex case. 
Let $\mu>0$ be fixed and consider the following function 
\begin{align*}
F_k(\bw) \coloneqq \frac{\mu}{4}\left(\frac12\inner{A_k\bw, \bw}- 
\inner{\bepsilon_1,\bw}\right),
\end{align*}
where $A_k$ is a $d\times d$ matrix defined as follows
\begin{align*} 
A = \left(\begin{matrix}
2& -1 & 0 & 0& \cdots & 0 & 0&0\\
-1& 2 & -1 & 0&\cdots & 0 & 0&0\\
\cdots&\cdots&\cdots&\cdots&\cdots&\cdots&\cdots&\cdots& 0_{k,d-k}\\
0& 0& 0 &0& \cdots & -1 & 2&-1\\
0& 0& 0 &0& \cdots & 0 & -1 &2\\
 &  &  &  & 0_{d-k,k}& &   & & 0_{d-k,d-k}
\end{matrix}\right),
\end{align*}
where $0_{m,n}$ is an $m\times n$ zero matrix. Given an iteration number $k$ 
such that $\tau+1\le k\le \frac12(d-1)(\tau+1)$, we take our function $F$ to be 
$F_{2\lfloor \frac{k}{\tau+1}\rfloor+1}(\bw)$. Using  
\lemref{proof:lem:non_zero_entries}, the only (possibly) non-zero entries of 
the $k$'th iterate $\bw_k$ are the first $\lfloor {k}/{(\tau+1)} \rfloor $ 
coordinates. Thus, following the same lines of proof as in  \cite[Theorem 
2.1.6]{nesterov2004introductory} yields
\begin{align*}
\frac{F(\bw_k) - F(\bw^*)}{\norm{\bw^*}^2}\ge 
\frac{3L}{32(k/(\tau+1)+1)^2}\ge 
\frac{L(\tau+1)^2}{45k^2}~.
\end{align*}

\section{Proof of \thmref{thm:stochastic_rate}} \label{proof:thm:stochastic_rate}

We will first state and prove the following auxiliary lemma:
\begin{lemma} \label{lem:convex_power_bound}
	The following holds for any $\eta >0$:
	\begin{itemize}
		\item For any $k\ge1$,
		\begin{align*}
		\max_{\{a~:~0 <a<1/\eta\} }a(1-\eta a)^k \le \frac{1}{e\eta k},
		\end{align*}
		where $e=2.718...$ is Euler's number. In particular, $\sum_{i=0}^k 
		\max_{\{a~:~0<a<1/\eta\} }a(1-\eta 
		a)^{2(i+1)} \le \frac{1}{2e\eta }H_k \le \frac{1}{2e\eta 
		}(1+\ln(k+1))$, where $H_k$ denotes the $k$'th harmonic number.
		\item If, in addition, we assume that $a>\lambda$ for some  
		constant $\lambda>0$, then 
		\begin{align*}
		\sum_{i=0}^k \max_{\{a~:~\lambda< a<1/\eta\} }a(1-\eta a)^{2(i+1)} \le 
		\frac{1+e+\ln( \frac{1}{\eta\lambda}) }{e\eta}.
		\end{align*}

	\end{itemize}
\end{lemma}
\begin{proof}
	By the well-known inequality $1+x\le \exp(x),~x\in\reals$, and since for 
	the domain over which we optimize it holds that $1-\eta a >0$, we have for 
	any $k\ge 1$
	\begin{align*}
	a(1-\eta a)^k\le a\exp(-\eta a k).
	\end{align*}
	Let us denote the latter by $\psi(a) \coloneqq a\exp(-\eta ak)$, and derive 
	for it the desired upper bound. 
	
	Taking the derivative of $\psi$ and setting to zero, gives 
	\begin{align*}
	(1-a \eta k)\exp(-\eta a k) =0.
	\end{align*}
	Therefore, the only stationary point of $\psi$  is $a^*=\frac{1}{\eta k }$. 
	Since $\psi'$ is positive for $a<a^*$ and negative for $a>a^*$, it follows 
	that $a^*$ is a global maximum, at which the value of $\psi$ is 
	$\frac{1}{e\eta k}$,  concluding the first part of the proof.
	
	Now, let $\lambda>0$. Since, the only maximizer of $\psi$ is at $a=\frac{1}{\eta 
	k}$, if $\lambda \ge\frac{1}{2\eta (i+1)}$, or equivalently 
	$i\ge\frac{1}{2\eta\lambda}-1$, then $\max_{\{a~:~\lambda< a<1/\eta\} 
	}a(1-\eta a)^{2(i+1)}\le \lambda(1-\eta \lambda)^{2(i+1)}$. Therefore,
	\begin{align*}
	\sum_{i=1}^k \max_{\{a~:~\lambda< a<1/\eta\} }a(1-\eta a)^{2(i+1)} &\le
	\sum_{i=1}^{\lfloor \frac{1}{2\eta\lambda}-1\rfloor  } \max_{\{a~:~\lambda< 
	a<1/\eta\} }a(1-\eta a)^{2(i+1)}\\
	&+ 
	\sum_{i=\lceil \frac{1}{2\eta\lambda}-1\rceil }^{k } \max_{\{a~:~\lambda< 
	a<1/\eta\} }a(1-\eta a)^{2(i+1)}\\
	&\le
	\sum_{i=1}^{\lfloor \frac{1}{2\eta\lambda}-1\rfloor } \frac{1}{e\eta k}
	+ 
	\sum_{i=\lceil \frac{1}{2\eta\lambda}-1\rceil }^{k } \lambda(1-\eta \lambda)^{2(i+1)}
	\\
	&\le \frac{1}{e\eta }(1+\ln( \frac{1}{2\eta\lambda})) + \frac{1}{\eta}\\
	&\le  \frac{1+e+\ln( \frac{1}{\eta\lambda}) }{e\eta}
	\end{align*}
	
\end{proof}

We now turn to prove \thmref{thm:stochastic_rate} itself. 
By Ineq. \ref{ineq:stochastic}  we have	
\begin{align} \label{ineq:stochastic_rate1}
2\E[F(\bw_k)-F(\bw^*)]&\le 	\|\sqrt{A}[z^k] D\|^2\|\be_0\|^2  + 
\eta^2\sigma^2\sum_{i=0}^k \left\|\sqrt{A}[z^i]D  \right\|^{2}.
\end{align}
We will bound each of the terms above separately. Assuming $\eta 
\in\left(0,\frac{1}{20\mu(\tau+1)}\right]$ we have by Ineq. 
\ref{ineq:generating_bound} and Ineq. \ref{ineq:coefficients_bound}, 
\begin{align} \label{ineq:stochastic_rate0}
\begin{aligned}
\|\sqrt{A}[z^k] D\|^2 &=
\|\sqrt{A}[z^k]\left((I-Iz+\eta A z^{\tau+1})^{-1}\right)\|^2\\
&\le \max_{i\in[d]}\left|\sqrt{a_i}[z^k] \frac{1}{\pi_{\eta a_i}(z)} \right|^2\\
&\le
\begin{cases}
\max_{i\in[d]} a_i & 0\le k \le (\tau+1)\ln(2(\tau+1))-1,\\
9\max_{i\in[d]} a_i (1-\alpha)^{2(k+1)}& k \ge (\tau+1)\ln(2(\tau+1)),
\end{cases}
\end{aligned}
\end{align}
Thus, for the first term, assuming $k \ge (\tau+1)\ln(2(\tau+1))$, we have
\begin{align} \label{ineq:stochastic_rate2}
\begin{aligned}
\|\sqrt{A}[z^k] D\|^2 &\le 9  \max_{i\in[d]} a_i (1-\eta a_i)^{2(k+1)}
\le 9 \mu \max_{i\in[d]}  (1-\eta a_i)^{2(k+1)}\\
&\le 9 \mu \exp(-2\eta \lambda(k+1)).
\end{aligned}
\end{align}
Bounding the second term in Ineq. \ref{ineq:stochastic_rate1} is somewhat more 
involved and requires separating into the two regimes stated in Ineq. \ref{ineq:stochastic_rate0}:
\begin{align}\label{ineq:stochastic_rate3}
\begin{aligned}
\sum_{i=0}^k \left\|\sqrt{A}[z^i]D  \right\|^{2}
&\le
\sum_{i=0}^{\lceil(\tau+1)\ln(2(\tau+1))\rceil-1} \left\|\sqrt{A}[z^i]D  \right\|^{2}
+\sum_{i=\lceil(\tau+1)\ln(2(\tau+1))\rceil}^{k} \left\|\sqrt{A}[z^i]D  \right\|^{2}\\
&\le \mu(\tau+1)\ln(2(\tau+1))
+9\sum_{i=0}^{k}  \max_{i\in[d]} a_i(1-\eta a_i)^{2(i+1)}
\end{aligned}
\end{align}
We proceed by considering the strongly convex case and the convex case 
separately. For the strongly convex case we have by 
\lemref{lem:convex_power_bound} 	
\begin{align*}
\sum_{i=0}^k \left\|\sqrt{A}[z^i]D  \right\|^{2} 
&\le  \mu(\tau+1)\ln(2(\tau+1))
+9\sum_{i=0}^{k}  \max_{i\in[d]} a_i(1-\eta a_i)^{2(i+1)}\\
&\le 
\mu(\tau+1)\ln(2(\tau+1))
+\frac{1+e+\ln(\frac{1}{\eta\lambda})}{\eta}.
\end{align*}
Together with Ineq. \ref{ineq:stochastic_rate1} and Ineq. \ref{ineq:stochastic_rate2}, this implies that for $k \ge (\tau+1)\ln(2(\tau+1))$,
\begin{align*}
2\E[&F(\bw_k)-F(\bw^*)]\le 	\|\sqrt{A}[z^k] D\|^2\|\be_0\|^2  + 
\eta^2\sigma^2\sum_{i=0}^k \left\|\sqrt{A}[z^i]D  \right\|^{2}\\
&\le 9 \mu \exp(-2\eta \lambda(k+1))\|\be_0\|^2 + \eta^2\sigma^2\left(\mu(\tau+1)\ln(2(\tau+1))
+\frac{1+e+\ln(\frac{1}{\eta\lambda})}{e\eta}\right)~,
\end{align*}
resulting in the first bound stated in the theorem. 
To get the second bound, we show how to 
optimally tune the step size $\eta$ (up to log factors).
Ignoring the log factors, the bound above is
\[
\E\left(F(\bw_k)-F(\bw^*)\right)~\leq~ 
\tilde{\Ocal}\left(\mu\norm{\be_0}^2\exp(-2\eta\lambda 
k)+\eta^2\sigma^2\left(\mu\tau+\frac{1}{\eta}\right)\right)~.
\]
Moreover, since we assume that $\eta\leq \Ocal(1/\mu\tau)$, we get that 
$\mu\tau$ is dominated (up to constants) by $1/\eta$, so we can simplify the 
above to
\begin{equation}\label{eq:strconvtilde}
\E\left(F(\bw_k)-F(\bw^*)\right)~\leq~ 
\tilde{\Ocal}\left(\mu\norm{\be_0}^2\exp(-2\eta\lambda 
k)+\eta\sigma^2\right)~.
\end{equation}
We now consider three cases:
\begin{itemize}
	\item If $0\leq \frac{\ln(\lambda\mu\norm{\be_0}^2 k/\sigma^2)}{2\lambda 
	k}\leq 
	\frac{1}{20(\mu\tau)}$, we can pick $\eta = \frac{\ln(\lambda 
	\mu\norm{\be_0}^2
	k/\sigma^2)}{2\lambda k}$, and get that \eqref{eq:strconvtilde} is
	\begin{align*}
	\tilde{\Ocal}\left(\frac{\sigma^2}{\lambda k}\right)~=~
	\tilde{\Ocal}\left(\mu\norm{\be_0}^2\exp\left(-\frac{\lambda k
	}{10\mu\tau}\right)+\frac{\sigma^2}{\lambda k}\right)
	\end{align*}
	\item If $ \frac{\ln(\lambda\mu\norm{\be_0}^2 k/\sigma^2)}{2\lambda 
		k} <0$, it follows that $\mu\norm{\be_0}^2\leq 
		\frac{\sigma^2}{\lambda k}$. In that case, we pick $\eta=0$, and get 
		that \eqref{eq:strconvtilde} is
		\[	
		\tilde{\Ocal}\left(\mu\norm{\be_0}^2\right)~\leq~\tilde{\Ocal}\left(\frac{\sigma^2}{\lambda
		 k}\right)~=~
	 \tilde{\Ocal}\left(\mu\norm{\be_0}^2\exp\left(-\frac{\lambda k
 }{10\mu\tau}\right)+\frac{\sigma^2}{\lambda k}\right)~.
		\]
	\item If $\frac{\ln(\lambda\mu\norm{\be_0}^2 k/\sigma^2)}{2\lambda 
		k}> 
	\frac{1}{20(\mu\tau)}$, we pick $\eta=\frac{1}{20(\mu\tau)}$, and get that 
	\eqref{eq:strconvtilde} is
	\[
	\tilde{\Ocal}\left(\mu\norm{\be_0}^2\exp\left(-\frac{\lambda 
	k}{10\mu\tau}\right)+\frac{\sigma^2}{\mu\tau}\right)~\leq~
\tilde{\Ocal}\left(\mu\norm{\be_0}^2\exp\left(-\frac{\lambda k
}{10\mu\tau}\right)+\frac{\sigma^2}{\lambda k}\right)~.
	\]
\end{itemize}
Collecting the three cases above, we get a 
bound of
\[
\tilde{\Ocal}\left(\mu\norm{\be_0}^2\exp\left(-\frac{\lambda k
}{10\mu\tau}\right)+\frac{\sigma^2}{\lambda k}\right)
\]
as required.


For the convex case, we have by Ineq. \ref{ineq:stochastic_rate1}, Ineq.  \ref{ineq:stochastic_rate0} and \lemref{lem:convex_power_bound}, that for $k \ge (\tau+1)\ln(2(\tau+1))$
\begin{align*}
\E[&F(\bw_k)-F(\bw^*)]\\
&\le 
 \frac{9}{4e\eta(k+1)}\|\be_0\|^2 + \frac{\eta^2\sigma^2}{2}\left(\mu(\tau+1)\ln(2(\tau+1))
+\frac{9}{2e\eta }(1+\ln(k+1))\right)~,
\end{align*}
resulting in the third bound in the theorem. To get the fourth bound, we now 
show how to optimally tune the step size $\eta$ (up to log factors).
Ignoring the log factors, the bound above is
\[
\tilde{\Ocal}\left(\frac{\norm{\be_0}^2}{\eta 
k}+\eta^2\sigma^2\left(\mu\tau+\frac{1}{\eta}\right)\right)~.
\]
As in the strongly convex case, since we assume $\eta\leq \Ocal(1/(\mu\tau)$, 
we 
can simplify the above to
\[
\tilde{\Ocal}\left(\frac{\norm{\be_0}^2}{\eta 
	k}+\eta\sigma^2\right)~.
\]
We now consider two cases:
\begin{itemize}
	\item If $\frac{\norm{\be_0}}{\sigma\sqrt{k}}\leq \frac{1}{20(\mu\tau)}$, 
	we choose $\eta=\frac{\norm{\be_0}}{\sigma\sqrt{k}}$, and get
	\[
	\tilde{\Ocal}\left(\frac{\norm{\be_0}\sigma}{\sqrt{k}}\right)~=~
	\tilde{\Ocal}\left(\frac{\norm{\be_0}^2\mu\tau}{
		k}+\frac{\norm{\be_0}\sigma}{\sqrt{k}}\right)~.
	\]
	\item If $\frac{\norm{\be_0}}{\sigma\sqrt{k}}> \frac{1}{20(\mu\tau)}$, we 
	choose $\eta=\frac{1}{20(\mu\tau)}$, and get
	\[
	\tilde{\Ocal}\left(\frac{\norm{\be_0}^2\mu\tau}{
	k}+\frac{\sigma^2}{\mu\tau}\right)~\leq~
	\tilde{\Ocal}\left(\frac{\norm{\be_0}^2\mu\tau}{
		k}+\frac{\norm{\be_0}\sigma}{\sqrt{k}}\right)~.
	\]
\end{itemize}
Collecting the two cases above, we get a 
bound of
\[
\tilde{\Ocal}\left(\frac{\norm{\be_0}^2\mu\tau}{
k}+\frac{\norm{\be_0}\sigma}{\sqrt{k}}\right)
\]
as required.
%
%

\end{document}